\documentclass{amsart}

\usepackage{amssymb,amsmath,amsfonts,bm,latexsym,amscd,verbatim,colonequals,url,nicefrac}
\usepackage[all]{xy}
\usepackage[letterpaper, left=3cm, right=3cm, height=22cm]{geometry}
\usepackage[latin1]{inputenc}
\usepackage[T1]{fontenc}
\usepackage{times}
\usepackage{graphicx}

\newtheorem{thm}{Theorem}{\bf}{\it}
\newtheorem{warning}[thm]{Warning}{\bf}{}
\newtheorem{prop}[thm]{Proposition}
\newtheorem{lem}[thm]{Lemma}
\newtheorem{cor}[thm]{Corollary}
\newtheorem{df}[thm]{Definition}{\bf}{}
\newtheorem{exam}[thm]{Example}{\bf}{}

\newcommand{\nn}{\mathbb{N}}

\newcommand{\zz}{\mathbb{Z}}

\newcommand{\catop}{\cat\op}
\newcommand{\cat}{\mathbf{C}}%{\mathscr{C}}
\newcommand{\ox}{\mathcal{O}_X}
\newcommand{\fu}{\mathbb{F}_1}
\newcommand{\ra}{\rightarrow}

\newcommand{\mcf}{\mathcal{F}}
\newcommand{\mcg}{\mathcal{G}}

\newcommand{\mco}{\mathcal{O}}

\newcommand{\mfm}{\mathfrak{m}}
\newcommand{\mfp}{\mathfrak{p}}

\DeclareMathOperator{\colim}{colim}

\DeclareMathOperator{\Hom}{Hom}
\DeclareMathOperator{\HHom}{\bf{Hom}}
\DeclareMathOperator{\op}{^{op}}

\DeclareMathOperator{\Spec}{Spec}

\DeclareMathOperator{\Aff}{\bf{Aff}}
\DeclareMathOperator{\Alg}{\bf{-Alg}}

\DeclareMathOperator{\Mod}{\bf{-Mod}}
\DeclareMathOperator{\Mon}{\bf{Mon}}
\DeclareMathOperator{\MS}{\bf{MS}}
\DeclareMathOperator{\Ring}{\bf{Ring}}
\DeclareMathOperator{\Psh}{\bf{Psh}}
\DeclareMathOperator{\PshAff}{\bf{Psh(Aff)}}
\DeclareMathOperator{\Sch}{\bf{Sch}}
\DeclareMathOperator{\Set}{\bf{Set}}

\begin{document}

\title{Deitmar's versus To\"en-Vaquié's schemes over $\fu$}
\author{Alberto Vezzani}
\address{Department of Mathematics\\ Università degli Studi di Milano\\ Via C. Saldini 50\\I-20133 Milan\\ Italy}
\email{alberto.vezzani@unimi.it}

\begin{abstract}
Deitmar indtroduced schemes over $\fu$, the so-called ``field with one element'', as certain spaces with an attached sheaf of monoids, generalizing the definition of schemes as ringed spaces. On the other hand, To\"en and Vaquié defined them as particular Zariski sheaves over the opposite category of monoids, generalizing the definition of schemes as functors of points. We show the equivalence between Deitmar's and To\"en-Vaquié's notions and establish an analog of the classical case of schemes over $\zz$. 
This result has been assumed by the leading experts on $\fu$, but no proof was given. During the proof, we also conclude some new basic results on commutative algebra of monoids, such as a characterization of local flat epimorphisms and of flat epimorphisms of finite presentation. We also inspect the base-change functors from the category of schemes over $\fu$ to the category of schemes over $\zz$.

\end{abstract}

\maketitle

\section*{Introduction}

Although the ``field with one element'' was originally mentioned in 1956 by Tits \cite{tits}, it in fact emerged as an significant object to investigate in the '90s. Despite its youth, a lot of interesting constructions have been built out of studying $\fu$-geometry, especially in the last decade.
The interested reader may find excellent commentaries on the motivations of this theory in various papers, such as \cite{cohn}, \cite{deitmar}, \cite{durov}. We also refer to the beautiful article of J.\ L\'opez Pe\~na and O.\ Lorscheid \cite{penalorscheid}, in which the whole picture of the $\fu$-universe is presented. The $\fu$-geometry project has been considered too ambitious by many, since none of the big aims that motivated its introduction has been reached yet. That said, we have to specify that the theory itself has not been settled fully since a lot of different approaches 
have been made, and thus, it is still undergoing a continuous evolution. Moreover, it seems that some results in other parts of mathematics, such as combinatorics, can really be proven using the $\fu$-machinery. We also feel that some of the approaches to $\fu$-geometry, such as the ones we present in here, are undoubtedly elegant as well as natural, being in turn relevant on their own.

In this paper, we focus mainly on Deitmar's and To\"en-Vaquié's theory. The reason for this is that we show their equivalence, generalizing a classical result of Demazure and Gabriel (\cite{demazuregabriel} I.1.4.4) to $\fu$-geometry (Theorem \ref{thm:main2}). Indeed, this has been taken for granted by many (see the map in \cite{penalorscheid}), but only partial results were given. In particular, we find that the core of this fact (which is Theorem \ref{thm:3opens}), despite having a rather elementary proof, is not trivial. This result is strongly related to some facts on commutative monoids that generalize similar statements on commutative rings. However, the tools to be used are necessarily different. For instance, this is because the category of $M$-modules for a given monoid $M$ is not an abelian category. In developing such theory, we were hugely inspired by the classical duality of schemes: they can be seen either as ``geometrical'' beings - ringed spaces which are locally affine, or as ``functorial'' beings - Zariski sheaves on the opposite category of rings, which are locally affine. Our result can be generalized as a new proof of this equivalence that only partly overlaps with the classical one of Demazure and Gabriel.

\section*{Notation}
In all this work, a choice of a universe $\mathcal{U}$ is implicit, and all the categories we introduce must be thought as $\mathcal{U}$-small categories (see also \cite{schapira}, 1.1, 1.2). We indicate categories with bold fonts. The category of sets is denoted by $\Set$. For a given category $\cat$ and an object $X$ inside it, we write $\Psh(\cat)$ for the category $\Set^{\catop}$ of presheaves over $\cat$, $\cat_{/X}$ for the category of objects over $X$, and $^{X/}\!\cat$ for the category of objects under $X$. 

The word ``ring'' will indicate a commutative ring with unity unless otherwise specified. Also, maps of rings respect the unity elements, hence subrings have the same unity of the bigger ring. The category of rings will be denoted by $\Ring$.

Similarly, the word ``monoid'' will indicate a commutative monoid unless otherwise specified. The category of monoids will be denoted by $\Mon$.

A closed symmetric monoidal category in the sense of \cite{kelly} will be indicated with $(\cat,\otimes)$ omitting all the extra structure, the unit object will be indicated with $\mathbf{1}$ and the internal Hom functor with $\HHom$. The category of monoids in $(\cat,\otimes)$ will be denoted by $\Mon_\cat$. For a given monoid $A$ in $(\cat,\otimes)$, the category of modules over $A$ will be indicated with $A\Mod$, the category $^{A/}\!\Mon_\cat$ will be denoted by $A\Alg$ and its objects will be called $A$-algebras.

\section{Schemes over $\fu$ à la Deitmar}

The following definitions were presented by Kato in \cite{kato} and Deitmar in \cite{deitmar}. In the latter paper, the author shows that the operation of the sum in rings can be overlooked for many purposes, and some of the basic notions and facts of algebraic geometry can be straightforwardly generalized to a broader context. 

\begin{df}
In a monoid $M$, a subset $I$ is an \emph{ideal} if the set \[IM:=\{xm\colon x\in I, m\in M\}\] equals $I$, and it is \emph{prime} if $M\setminus I$ is a submonoid of $M$. The \emph{prime spectrum of $M$ over $\fu$} is the topological set of all prime ideals $\mfp$ of $M$, with the topology in which closed sets are of the form $V(I):=\{\mfp\colon I\subset\mfp\}$, where $I$ is a subset of $M$. It is indicated with $\Spec_{\fu}(M)$ (or simply with $\Spec M$ if the context is clear) and its topology is called the \emph{Zariski topology}. 
\end{df}

We can say that every monoid $M$ is local, in the sense that it has a unique maximal proper ideal, namely the subset of non-invertible elements $M\setminus M^\times$. It is obviously a prime ideal, and it is the only closed point of $\Spec M$. We also remark that $\Spec M$ has a basis of open sets constituted by the empty set and those of the form $D(a):=\{\mfp\colon a\notin\mfp\}$, where $a$ is an element of $M$. An open subset $D(x)$ is never empty since it contains the point $\emptyset$. In particular, since $D(a)\cap D(b)=D(ab)$, the space $\Spec M$ is irreducible. Also, we remark that every open covering includes the open subset $\Spec M$ itself, since the only open $D(a)$ that contains the maximal ideal is $D(1)=\Spec M$.

\begin{df}
A map $f\colon M\ra N$ of monoids is \emph{local} if $f(M\setminus M^\times)\subset N\setminus N^\times$, i.e. if $f^{-1}( N^\times)=M^\times$.
\end{df}

One of the main special features of prime spectra of rings is the structure sheaf, defined via localizations. Also in this setting, localizations can be defined using similar techniques. 

\begin{df}
For a subset $S$ of $M$, we call \emph{localization of $M$ at $S$} the monoid $S^{-1}M$ with a map $\pi\colon M\ra S^{-1}M$ which has the following universal property: for every map of monoids $f\colon M\ra N$ such that $f(S)\subset N^\times$, there exists a unique map $S^{-1}M\ra N$ that splits $f$ over $\pi$. If $S=\{a\}$, we indicate $S^{-1}M$ with $M_a$. If $S=M\setminus\mfp$ where $\mfp$ is a prime ideal, we indicate $S^{-1}M$ with $M_\mfp$. 
\end{df}

We remark that if two elements of $M$ are sent to units in $N$, so is their product. Also, the unity of $M$ is always mapped to the unity of $N$. We can then restrict ourselves to considering localizations with respect to submonoids $S$ of $M$. The result \cite{atmc} 3.1 can be generalized to prove that any localization $S^{-1}M$ is well defined, and has the following explicit description: as a set, $S^{-1}M$ is the set of formal fractions
\[
\left.\left\{\frac{a}{x}\colon a\in M, x\in S\right\}\right/\!\!\sim
\]
where $\frac{a}{x}\sim\frac{b}{y}$ if there exists an element $t\in S$ such that $ayt=bxt$. The monoid operation in $S^{-1}M$ is defined as $\frac{a}{x}\cdot\frac{b}{y}=\frac{ab}{xy}$ and the map of monoids $M\ra S^{-1}M$ is the map $a\mapsto\frac{a}{1}$.

\begin{df}
A \emph{monoidal space} is a pair $(X,\ox)$ consisting of a topological space $X$ and a sheaf of monoids $\ox$ on it. A morphism of monoidal spaces from $(X,\ox)$ to $(Y,\mco_Y)$ is a pair $(f,f^\sharp)$ where $f\colon X\ra Y$ is a continuous map and $f^\sharp\colon\mathcal{O}_Y\ra f_*\ox$ is a map of sheaves on $Y$ such that for every $x\in X$, the induced morphism of stalks \mbox{$f^\sharp_x\colon\mco_{Y,f(x)}\ra\mco_{X,x}$} is local. The category of monoidal spaces is denoted by $\MS$. 
\end{df}

\begin{prop}\label{prop:mscoc}
The category $\MS$ is cocomplete.
\end{prop}

\begin{proof}
The proof is the exact analogue of \cite{demazuregabriel}, I.1.1.6.
\end{proof}

\begin{prop}\label{prop:specadjmon}
Let $M$ be a monoid. There is a canonical structure of monoidal space on $\Spec_{\fu}\!\!M$ such that $\Spec_{\fu}$ defines a left adjoint of the functor of global sections $\Gamma$, seen as a functor from $\MS\op$ to $\Mon$. In particular, for any monoidal space $(X,\mco_X)$
\[
\Hom_{\Mon}(M,\Gamma(X,\mco_X))\cong\Hom_{\MS}(X,\Spec_{\fu}\!\! M).
\]
The sheaf $\mco_{\Spec_{\fu}\!\! M}$ is such that $\mco_{\Spec_{\fu}\!\! M}(D(a))=M_a$ for any element $a$ in $M$ and  \mbox{$\mco_{\Spec_{\fu}\!\! M,\mfp}=M_\mfp$} for any prime ideal $\mfp$ of $M$.
\end{prop}

\begin{proof}
The proof is the exact analogue of \cite{demazuregabriel}, I.1.2.1.
 
\end{proof}

\begin{df}
Monoidal spaces which are isomorphic to $(\Spec_{\fu}\!\! M,\mco_{\Spec_{\fu}\!\! M})$ for some monoid $M$ are called \emph{affine geometrical \mbox{$\fu$-schemes}}. 
\end{df}

The previous proposition implies in particular that the functor $\Spec_{\fu}$ from monoids to affine geometrical $\fu$-schemes is part of a contravariant equivalence of categories.

\begin{df}\label{df:d-zaraff}
A map $(X,\mco_X)\ra(Y,\mco_Y)$ of $\MS$ is an \emph{open immersion} if it is the composite of an isomorphism and an open inclusion $(U,\mco_Y|_U)\hookrightarrow(Y,\mco_Y)$. A family of open immersions is a \emph{Zariski covering} if it is globally surjective on the topological spaces underneath. A \emph{geometrical $\fu$-scheme} (or \emph{scheme over $\fu$ à la Deitmar}) is a monoidal space $(X,\mathcal{O}_X)$ with an affine Zariski covering. The \emph{category of geometrical $\fu$-schemes} is the full subcategory of $\MS$ whose objects are geometrical $\fu$-schemes. 
It is easy to prove that Zariski coverings define a Grothendieck pretopology in the category of geometrical $\fu$-schemes. The site they form is called the \emph{Zariski site}.
\end{df}
The category of geometrical $\fu$-schemes is not cocomplete. Still, it has some colimits. In particular, it is straightforward to generalize the gluing lemma (\cite{hartshorne}, Exercise II.2.12) to this context. 

\begin{prop}\label{prop:zarissub}
The Zariski topology on geometrical $\fu$-schemes is subcanonical. Also, the category of affine geometrical $\fu$-schemes is dense in the category of $\fu$-schemes, in the sense that each geometrical $\fu$-scheme is a colimit of a diagram contained in the subcategory of affine geometrical $\fu$-schemes.
\end{prop}

\begin{proof}
Suppose that $\{U_i=\Spec M_i\ra X\}$ is a Zariski covering of $X$. Let $\{\Spec A_{ijk}\ra U_i\cap U_j\}$ be coverings of the schemes $U_i\cap U_j$. Then the following are coequalizing diagrams
\[\coprod U_i\cap U_j\rightrightarrows\coprod U_i\ra X\]
\[\coprod \Spec A_{ijk}\rightrightarrows\coprod \Spec M_i\ra X\]
and this implies the claim. 
 
\end{proof}

As in the case of ordinary schemes, the category of geometrical $\fu$-schemes has pullbacks (also called fibered products), and affine geometrical $\fu$-schemes are closed under pullbacks (\cite{deitmar}, 3.1).

In the classical case of schemes, the spectrum of a ring can be defined though a colimit using $K$-points, as $K$ varies among the fields (\cite{demazuregabriel}). In the case of monoids, the naive attempt would be to consider the $G$-points as $G$ runs through the category of groups. This does not work, as the following remark specifies.

\begin{prop}
Let $G$ be an abelian group and $X$ a monoidal space. Defining a $G$-point on $X$ is the same as considering a point $x$ of $X$ such that $\mco_{X,x}$ is a group, together with a group homomorphism $\mco_{X,x}\ra G$.
\end{prop}

\begin{proof}
Suppose that $f$ is a map from $\Spec_{\fu}\!\! G$ to $X$. Since a group has only one prime ideal $\emptyset$, the map $f$ defines automatically a point $x=f(\emptyset)$ in $X$. Adding to this, it defines a local map of monoids $\mco_{X,x}\ra G$. The fact that this map is local implies that all elements of $\mco_{X,x}$ are invertible, as wanted.
Conversely, given a point $x$ and a homomorphism $\mco_{X,x}\ra G$, we can define a map between topological spaces that sends the unique point of $\Spec_{\fu}\!\!G$ to $x$. Note that the map $\mco_{X,x}=\varinjlim_{x\in U}\mco_X(U)\ra G$ induces maps $\mco_X(U)\ra G$ for every $U$ such that $x\in U$. Together with the trivial maps $\mco_X(U)\ra1$ for those open subsets $U$ that do not contain $x$, they define a map of sheaves $\mco_X\ra f_*\Spec_{\fu}\!\! G$, as wanted.
 
\end{proof}

In particular, we conclude that $G$-points on monoidal spaces are rare to find, so that there is no possibility to recover the topological space beneath just by using them. 

\section{Schemes over $\fu$ à la To\"en-Vaquié}
We now present the generalization of the concept of schemes introduced by To\"en and Vaquié in their paper \cite{toenvaquie}. One of the main advantages of this approach is its generality. The way new schemes are introduced is purely categorical and the case of $\fu$ is just a particular case of a more general picture, in which the protagonists are well-behaved monoidal categories.

From now on, we will consider a closed symmetric monoidal category $(\cat,\otimes)$ with unit $\mathbf{1}$ and inner Hom functor $\HHom$, which is complete and cocomplete. We know in particular that the tensor product commutes with colimits, because it has a right adjoint.

\begin{df}
Let $A$ be an object of $\Mon_\cat$, and let $M$, $N$ be objects of $A\Mod$ with actions $\varphi$, $\psi$ respectively. We define the \emph{tensor product of $M$ and $N$ over $A$}, and we indicate it with $M\otimes_A N$, the coequalizer in the diagram 
$$\xymatrix{
A\otimes M\otimes N\ar@<0.5ex>[r]^{\varphi\otimes N}\ar@<-0.5ex>[r]_{\psi\otimes M}&M\otimes N.
}$$
It has a natural $A$-module structure.
\end{df}

It is easy to prove the following sequence of facts.

\begin{prop}\label{prop:factsontensor}
Consider a map $f\colon A\ra B$ in $\Mon_\cat$.
\begin{enumerate}
	\item There is a natural forgetful functor $B\Mod\ra A\Mod$ that sends an object $N$ to $N$ itself, considered as a $A$-module with the action defined as the composite
	\[A\otimes N\ra B\otimes N\ra N.\]
	In particular, the map $f$ defines a natural structure of \mbox{$A$-module} on $B$, with the action defined as above.
	\item The forgetful functor has a left adjoint, indicated with $\otimes_A B$, which sends a \mbox{$A$-module} $M$ to $M\otimes_A B$, with a suitable \mbox{$B$-action}.  
	\item The forgetful functor has a right adjoint, which sends a $A$-module $M$ to $\HHom(B,M)$, with a suitable $B$-action. 
	\item \label{push}The pushout in $\Mon_\cat$ of a diagram \mbox{$B\leftarrow A\ra C$} is isomorphic as $A$-module to $B\otimes_A C$.
\end{enumerate}
\end{prop}

In particular, for an object $A$ of $\Mon_\cat$, and for an object $M$ of $A\Mod$, $M\otimes_AA$ is canonically isomorphic to $M$ since both $\otimes_AA$ and the identity itself are left adjoint functors of the identity.

\begin{cor}\label{cor:adjalg}
Let $A\ra B$ be a map of $\Mon_\cat$. The forgetful functor $B\Alg\ra A\Alg$ has a left adjoint, which maps $A\ra X$ to $B\ra B\otimes_AX$ with the monoid structure induced by the isomorphism $B\otimes_AX\cong B\sqcup_AX$.
\end{cor}

\begin{df}\label{df:affsch}
The opposite category of the category of $\Mon_\cat$ is denoted by $\Aff_\cat$, and its objects are called \emph{affine schemes relative to $\cat$}. We call $\Spec A$ the object in $\Aff_\cat$ which corresponds to the monoid $A$ in $\Mon_\cat$.
\end{df}

It is now high time to introduce the Zariski topology on the category of affine schemes. 

\begin{df}\label{df:openimm1}
Suppose that $f\colon A\ra B$ is a map in $\Mon_\cat$. It is \emph{flat} if the functor $\otimes_A B$ from $A\Mod$ to $B\Mod$ is exact (equivalently, left exact) in the sense that it commutes with finite limits and colimits. The map $f$ is \emph{of finite presentation} if for every direct system $\{C_i\}_{i\in I}$ of $A$-algebras, the canonical map
	\[
	\varinjlim\Hom_{A\Alg}(B,C_i)\ra\Hom_{A\Alg}(B,\varinjlim C_i)
	\]
	is bijective.
A map $\Spec B\ra\Spec A$ is an \emph{open immersion} if the correspondent map $A\ra B$ is a flat epimorphism of finite presentation, and a collection of open immersions $\{\Spec A_i\ra\Spec A\}_{i\in I}$ is a \emph{Zariski covering} if there is a finite subset $J\subset I$ such that the collection $\{\Spec A_j\ra\Spec A\}_{j\in J}$ reflects isomorphisms of modules, in the sense that any map of $A$-modules \mbox{$M\ra N$} is an isomorphism if and only if the induced maps \mbox{$M\otimes_A A_j\ra N\otimes_A A_j$} are isomorphisms, for all $j\in J$.
\end{df}

It is easy to prove that Zariski coverings define a Grothendieck pretopology on $\Aff_\cat$, and the site they form is again called the \emph{Zariski site}.

\begin{prop}\label{prop:dgistoen}
In case $(\cat,\otimes)$ is the category of abelian groups with the tensor product $\otimes_\zz$, then the Zariski site on $\Aff_\cat$ is equivalent to the Zariski site on affine schemes.
\end{prop}

\begin{proof}
In this case the category $\Aff_\cat$ is the category $\Ring\op$, which is equivalent to the category of affine schemes because of \cite{EGAI} I.7.4. A map of rings $A\ra B$ induces an open immersion if and only if it is a flat epimorphism of finite presentation because of \cite{EGAIV4}, 17.9.1. Also, using \cite{atmc} 3.9, a collection $\{A\ra B_i\}$ induces a covering of $\Spec A$ if and only if it reflects isomorphisms of modules. Because any affine scheme is quasi-compact, it is always possible to extract a finite sub-covering labeled by $J$, and this proves the claim.
 
\end{proof}

Note that, in particular, it is part of the definition the fact that affine schemes are quasi-compact (a finite sub-covering is indexed by $J$), while that is granted by the explicit definition of the Zariski topology in the case of rings.
Now that we introduced a topology on affine schemes, we can study Zariski sheaves over affine schemes. In the case of rings, the functor represented by any affine scheme was also a sheaf. In this more general setting, this fact is still true, and it needs indeed a more elaborated proof (\cite{toenvaquie}, 2.11).

We then use the word ``affine scheme'' to refer both to objects $X$ of $\Aff_\cat$ and also to functors $h_X$ represented by them. 
In order to define a scheme, we still have to define open coverings of sheaves, so to have a good definition of ``being locally affine'' also for a sheaf. 

\begin{df}\label{df:openimm2}
A map $f\colon \mcf\ra h_X$ of Zariski sheaves over $\Aff_\cat$ is an \emph{open immersion} if it defines $\mcf$ as a subsheaf of $h_X$ and if there exists a family of open immersions $\{X_i\ra X\}_{i\in I}$ such that $\mcf$ is isomorphic over $h_X$ to the image of the induced map $\coprod_{i\in I}{h_{X_i}}\ra h_X$. More generally, a map $f\colon \mcf\ra \mcg$ of Zariski sheaves over $\Aff_\cat$ is an \emph{open immersion} if for every affine scheme $h_X$ over $\mcg$, the induced morphism $\mcf\times_{\mcg} h_X\ra h_X$ is an open immersion. A collection $\{\mcf_i\ra \mcf\}_{i\in I}$ of open immersions is a \emph{Zariski covering} if the induced map $\coprod_{i\in I}{\mcf_i}\ra \mcf$ is an epimorphism.
\end{df}

One should check that all the definitions given agree on affine schemes. This is again something completely not trivial (\cite{toenvaquie}, 2.14).

We are now ready to give the definition of a scheme in this new setting.

\begin{df}
A \emph{scheme relative to $\cat$} (or a \emph{scheme à la To\"en-Vaquié relative to $\cat$}) is a Zariski sheaf over affine schemes in the sense of Definition \ref{df:affsch}, which has a Zariski covering constituted of open immersions of affine schemes. The \emph{category of schemes relative to $\cat$} is the full subcategory of $\Psh(\Aff_\cat)$ whose objects are schemes relative to $\cat$.
\end{df}

As a side note, we remark that in case $(\cat,\otimes)$ is the category of abelian groups with the tensor product $\otimes_\zz$, then the category of schemes relative to $\cat$ is equivalent to the category of schemes as defined in \cite{demazuregabriel}, I.1.3.11. This comes from Proposition \ref{prop:dgistoen} and the fact that a a family of open immersions $\{\mcf_i\ra\mcf\}$ induces an  an epimorphism of Zariski sheaves $\coprod\mcf_i\ra\mcf$ if and only if it induces a surjection $\coprod\mcf_i(\Spec K)\ra\mcf(\Spec K)$ for all fields $K$ (see \cite{toencorso}, Lemma 4.2.1).

As it is shown in \cite{toenvaquie}, 2.18, the category of schemes relative to $\cat$ inside the category of Zariski sheaves is stable under disjoint unions and fibered products. This easily implies that Zariski coverings define a Grothendieck pretopology on schemes relative to $\cat$. The site they form is again called the \emph{Zariski site}. 

Up to now, we presented the whole picture of generalized schemes à la To\"en-Vaquié. It is now time to focus on schemes over $\fu$ which another special case of the general theory.

\begin{df}
A \emph{$\fu$-scheme} or a \emph{scheme over $\fu$} is a scheme relative to the monoidal category $(\Set,\times)$. The category of $\fu$-schemes is denoted with $\Sch_{\fu}$.
\end{df}

In particular, since monoids in $(\Set,\times)$ are just ordinary commutative monoids, the category $\Aff_\cat$ is the category $\Mon\op$. We will henceforth refer to it with $\Aff$. Also, for a fixed monoid $M$, the category of $M$-modules is the category of $M$-sets, i.e. sets with an action of $M$. It is not an abelian category, since the initial object $\emptyset$ it is not the final object $\{*\}$. We also note that for a couple of $M$-modules $S$ and $T$, $S\otimes_M T$ is the set $S\times T$ modulo the equivalence relation generated by the relation $(m\cdot s,t)\sim(s,m\cdot t)$. In case $S$ and $T$ are $M$-algebras, by Proposition \ref{prop:factsontensor}, the module $S\otimes_M T$ inherits a $M$-algebra structure, and it is isomorphic to $S\sqcup_MT$ in the category $M\Alg$.

\section{Deitmar - To\"en-Vaquié equivalence}

We now want to prove the equivalence of categories between the two different notions of \mbox{$\fu$-schemes} that we introduced so far. A large part of this section is dedicated to commutative algebra of monoids, in which we try to set up an environment which is similar to the classical one of commutative rings. We denote with $\fu$ the trivial monoid $\{1\}$.

\begin{prop}
Let $M$ be a monoid. The forgetful functor from $M\Alg$ to $\Mon$ has a left adjoint which sends a monoid $N$ to $M\times N$ with the natural $M$-action. In particular, the forgetful functor from $M\Alg$ to $\Set$ has a left adjoint that sends a set $S$ to the monoid 
\[
M[S]:=\{m\cdot s_1^{d_1}s_2^{d_2}\ldots s_k^{d_k}:k\in\zz_{\geq0}, m\in M, s_i\in S, d_i\in\zz_{\geq0}\}
\]
with the obvious operation and $M$-action. We shall indicate the monoid $M[\{x_1,\ldots,x_n\}]$ with $M[x_1,\ldots,x_n]$.
\end{prop}

\begin{proof}
The category of monoids is the category of $\fu$-algebras, and for any couple of monoids $M$ and $N$, we have $M\otimes_{\fu}N=M\times N$. The result then follows from Corollary \ref{cor:adjalg}.
 
\end{proof}

\begin{exam}
Consider the monoid $(\zz_{\geq1},\cdot)$. It is isomorphic to $\fu[x_1,x_2,\ldots]$ through the map $x_i\mapsto p_i$, where the $p_i$'s are the positive primes. 
\end{exam}

\begin{df}
Let $M$ be a monoid and let $\varphi\colon M\ra N$ be a $M$-algebra. An equivalence relation $\sim$ on $N$ is \emph{monoidal and $M$-linear} if it is defined by a subset of $N\times N$ which is a sub-$M$-algebra with respect to the diagonal action of $M$ on $N\times N$.
Given a monoidal $M$-linear equivalence relation $\sim$ on $N$, it is possible to define a structure of $M$-algebra on $N/\!\!\sim$ mapping $m$ to $[\varphi(m)]$. A $M$-algebra $N$ is called \emph{finitely generated} if there exists an integer $n$ and a surjective map of \mbox{$M$-algebras} from $M[x_1,\ldots,x_n]$ to $N$. Equivalently, if it is isomorphic as $M$-algebra to $M[x_1,\ldots,x_n]/\!\!\sim$ for a suitable monoidal $M$-linear equivalence relation $\sim$.
\end{df}

\begin{prop}
Let $N$ be a $M$-algebra. Then $N$ is of finite presentation if and only if it is isomorphic as a $M$-algebra to $M[x_1,\ldots,x_n]/\!\!\sim$ where the relation $\sim$ is a finitely generated sub-\mbox{$M[x_1,\ldots,x_n]$-algebra} of the monoid $M[x_1,\ldots,x_n]\times M[x_1,\ldots,x_n]$ i.e. $N$ is the coequalizer in the category of $M[x_1,\ldots,x_n]$-algebras of a diagram
\[
M[x_1,\ldots,x_n][y_1,\ldots,y_m]\rightrightarrows M[x_1,\ldots,x_n]
\]
for some suitable $n,m\in\nn$.
\end{prop}

\begin{proof}
The proof runs in the same way as in \cite{EGAIV3}, 8.14.2.2. The only difference is that instead of taking quotients over ideals, we now have to consider quotients over $M[x_1,\ldots,x_n]$-linear monoidal equivalence relations.
 
\end{proof}

Let $\{p_i,q_i\}_{i\in I}$ be elements of $M[S]$. From now on, we indicate with $(p_i=q_i)_{i\in I}$ the monoidal $M[S]$-linear equivalence relation on $M[S]$ generated by the couples $(p_i,q_i)$. 

\begin{df}
Let $M$ be a monoid. We call it a \emph{monoid with zero} if there exists an element $0$ such that $\{0\}$ is an ideal. Arrows between monoids with zero are arrows of monoids that send $0$ to $0$. We call the category they form with $\Mon_0$. The forgetful functor $\Mon_0\ra\Mon$ has a left adjoint that sends $M$ to $M_0:=M\sqcup\{0\}$, with the obvious operation.
\end{df}

\begin{exam}
The monoid $(\zz,\cdot)$ is isomorphic to the monoid \[\left(\fu[u,x_1,x_2,\ldots]\biggr/\left(u^2=1\right)\right)_0\] through the map $u\mapsto -1$, $x_1\mapsto p_1$, where the $p_i$'s are the positive primes. 
\end{exam}

\begin{cor}\label{cor:locmonfp}
A localization of a monoid over a finite set of elements is of finite presentation.
\end{cor}

\begin{proof}
We can reduce ourselves to consider the case in which we localize over a single element $a$. It is straightforward that $M_a=M[x]/(ax=1)$. We can then apply the previous proposition and conclude the claim.
 
\end{proof}

\begin{prop}\label{prop:locmonflat}
Localizations of monoids are flat.
\end{prop}

\begin{proof}
Let $T$ be a $M$-module. 
The $S^{-1}M$-module \mbox{$S^{-1}T\colonequals T\otimes_MS^{-1}M$} has the following alternative description. Its underlying set is
\[
S^{-1}T\colonequals\left.\left\{\frac{t}{s}\colon t\in T, s\in S \right\}\right/\!\!\sim
\]
where $\sim$ is the equivalence relation that identifies $\frac{t}{s}$ and $\frac{t'}{s'}$ if there exists an element $s''\in S$ such that $s''s'\cdot t=s''s\cdot t'$. The action of $S^{-1}M$ is defined by $\frac{m}{s}\cdot\frac{t}{s'}:=\frac{m\cdot t}{ss'}$.

Let now $S$ be a multiplicatively closed subset of $M$. We have to prove that the functor $\otimes_MS^{-1}M$ commutes with equalizers and finite products in the category of $M$-modules. In this category, both these limits are built over the limits in the category of sets, with the obvious $M$-action induced. 
Let now $T$ and $U$ be $M$-modules. It is easy to see that the map
\[
\begin{aligned}
S^{-1}(T\times U)&\ra S^{-1}T\times S^{-1}U\\
\frac{(t,u)}{s}&\mapsto \left(\frac{t}{s},\frac{u}{s}\right)\\
\end{aligned}
\]
defines an isomorphism of $M$-modules from $S^{-1}(T\times U)$ to $S^{-1}T\times S^{-1}U$, as wanted.

Also, for two arrows of $M$-modules $\varphi,\psi\colon T\rightrightarrows U$ whose equalizer is $E$, there is a natural map from $S^{-1}E$ to the equalizer $E'$ of the induced couple of arrows $S^{-1}T\rightrightarrows S^{-1}U$.  This maps sends the element $\frac{x}{s}$ in $S^{-1}E$ to $\frac{x}{s}$, seen as an element of $S^{-1}T$. This map is clearly injective. Suppose now that $\frac{t}{s}$ is in $E'$. This means that $\frac{\varphi(t)}{s}=\frac{\psi(t)}{s}$, hence that there exists an element $s'\in S$ such that 
\[
\varphi(s's\cdot t)=s's\cdot \varphi(t)=s's\cdot\psi(t)=\psi(s's\cdot t).
\]
We then conclude that $\frac{t}{s}=\frac{s's\cdot t}{s's^2}$ and $s's\cdot t\in E$. This proves the surjectivity, hence the claim.
 
\end{proof}

The following two results concern flat epimorphisms of monoids. In particular, we would like to conclude that local flat epimorphisms are isomorphisms. Stenstr\"om in \cite{stenstrom} refers to the work of Roos and he states that flat epimorphisms of (non necessarily commutative) monoids can be characterized as localizations over Gabriel topologies, using the tools of torsion theory developed in \cite{gabriel} by Gabriel. Indeed, any epimorphism of monoids $M\ra N$ induces a full embedding of categories $N\Mod\ra M\Mod$ via the forgetful functor. Due to the flatness property, this forgetful functor has also an exact left adjoint, hence it defines a localization of $M\Mod$. However, the proof of the fact that such reflective subcategories are all localizations with respect to some Gabriel topologies of monoids is not present in \cite{stenstrom}, and it is not a direct corollary of the general results of Gabriel, who considered abelian categories. Therefore, since in our case $M\Mod$ is not abelian, we prefer to follow a more explicit approach, which is in turn valid just for our specific setting. 

Analogous results on the comparison of the two topologies on $\Mon\op$ have been proven independently by Florian Marty, who used a more abstract and general approach, based on Gabriel filters. All the details can be found in his article \cite{martyopen}.

\begin{lem}\label{lemma:surjoninv}
A local epimorphism of monoids is surjective on invertible elements. 
\end{lem}

\begin{proof}
Let $\varphi\colon M\ra N$ be a local epimorphism of monoids. Consider the set $N/\!\!\sim_{\mfm}$, where $\sim_\mfm$ identifies the elements of the maximal ideal $\mfm:=N\setminus N^\times$. It has a natural monoid structure induced by the one in $N$, and it is isomorphic to the monoid with zero $(N^\times)_0$. We also consider the subgroup $\varphi(M^\times)$ in $N^\times$, and the quotient taken in the category of groups $T:=N^\times/\varphi(M^\times)$. We can now consider two maps $(N^\times)_0\rightrightarrows T_0$: the first one is induced by the projection, the second is induced by the constant map $N^\times\mapsto1_T$. Since $\varphi$ is local, the image of an element in $M$ via the two composite maps $N\ra (N^\times)_0\rightrightarrows T_0$ is the same. Hence, because $\varphi$ is an epimorphism, we conclude that $\varphi(M^\times)=N^\times$. 
 
\end{proof}

The statement of the following proposition is a generalization of a standard fact on the category of rings (see \cite{lazard}, IV.1.2). 

\begin{prop}\label{prop:flatepifpmon}
Let $\varphi\colon M\ra N$ be a map of monoids.
\begin{enumerate}
	\item If $\varphi$ is local and flat, then it is injective.
	\item If $\varphi$ is a local flat epimorphism, then it is an isomorphism.
\end{enumerate}
\end{prop}

\begin{proof}
We initially prove the first claim. Suppose that $\varphi(a)=\varphi(b)=t$. Consider the two maps of $M$-modules $M\ra M$, $1\mapsto a$ and $1\mapsto b$, and let $E$ be their equalizer. By using the isomorphisms of $M$-modules $m\otimes n\mapsto \varphi(m)n$ from $M\otimes N$ to $N$, we conclude that the two maps tensored with $N$ are both equal to the map $N\ra N$, $n\mapsto tn$. In particular, the equalizer of the two is the whole of $N$. By the flatness property, we then deduce that the map $E\otimes N\ra N$, $x\otimes n\mapsto \varphi(x)n$ is an isomorphism. In particular, there exists an element $x\in E$ and an element $n\in N$ such that $\varphi(x)n=1$. Because the map is local, we conclude that $x$ is invertible. Since $ax=bx$, this implies that $a=b$.

Now we turn to the second claim. Because we already know that $\varphi$ is injective, we consider $M$ as a submonoid of $N$, and consider $\varphi$ as the inclusion. We recall that a map is an epimorphism if and only if its cokernel pair is constituted by identities. Because $N\otimes_MN$ is the cokernel pair of $\varphi$ in the category of monoids (Proposition \ref{prop:factsontensor}), we conclude that the two maps $N\ra N\otimes_M N$ defined as \mbox{$n\mapsto 1\otimes n$} and \mbox{$n\mapsto n\otimes 1$} are isomorphisms. Now consider the $M$-module $N/\!\!\sim_M$, defined as the quotient of $N$ with respect to the equivalence relation which identifies the elements of $M$. It has a well-defined $M$-module structure induced by the one of $N$, and a natural projection map $\pi\colon N\ra N/\!\!\sim_M$. This projection has the following universal property: any map of \mbox{$M$-modules} $N\ra T$ such that the image of $M$ is constant, splits uniquely through $\pi$. In other words, $\pi$ is the pushout of the diagram below.
$$\xymatrix{
M\ar[d]\ar[r]^{\varphi}&N\\
\{*\}&
}$$
Because of the flatness property, $\otimes_MN$ commutes with small products, hence it preserves the terminal object $\{*\}$ (the empty product). Also, because it commutes with colimits and $\varphi\otimes_MN=id_N$, we conclude that $(N/\!\!\sim_M)\otimes_MN$ is the pushout of the diagram
$$\xymatrix{
N\ar[d]\ar[r]^{=}&N\\
\{*\}&
}$$
hence it is the trivial module $\{*\}$.

We now inspect the kernel pair $K$ of the projection \mbox{$\pi\colon N\ra N/\!\!\sim_M$}. It is constituted by the couples $(x,y)$ in $N\times N$ such that $\pi(x)=\pi(y)$. Since $(N/\!\!\sim_M)\otimes_M N$ is the terminal object, the kernel pair of the tensored map is the product of two copies of $N\otimes_MN=N$. Because of the flatness property, we then conclude that the map $K\otimes_MN\ra N\times N$, $(x,y)\otimes n\mapsto(xn,yn)$ is an isomorphism. Fix now an element $\bar{n}$ of $N$. In particular, the couple $(1,\bar{n})$ has to be reached by the previous map, hence there is a couple $(x,y)\in K$ and an element $n\in N$ such that $xn=1$ and $yn=\bar{n}$. We then conclude that $n$ and $x$ are invertible, hence they are elements of $M$ by Lemma \ref{lemma:surjoninv}. Because the couple $(x,y)$ lies in $K$ and $x$ is in $M$, we conclude that also $y$ is in $M$. Therefore, $\bar{n}$ is an element of $M$. This holds for any $\bar{n}$, hence $M=N$. We then showed that $\varphi$ is also surjective. Because any bijective map of monoids is an isomorphism, the claim is proven. 
 
\end{proof}

\begin{thm}\label{thm:3opens}
Let $\varphi\colon M\ra N$ be a morphism of monoids. The following are equivalent.
\begin{enumerate}
	\item \label{m1}The map $\varphi$ is a flat epimorphism, of finite presentation.
	\item \label{m2}The map $\varphi$ is isomorphic as a $M$-algebra to a localization over an element of $M$.
	\item \label{m3}The map $\varphi$ defines an open immersion of affine geometrical \mbox{$\fu$-schemes}.
\end{enumerate}
\end{thm}

\begin{proof}
The fact that (\ref{m2}) implies (\ref{m3}) is obvious. It is also easy to show that (\ref{m3}) implies (\ref{m2}). Indeed, suppose that $\Spec_{\fu}\!\! N$ is an open geometrical $\fu$-subscheme of $\Spec_{\fu}\!\! M$. Cover it with basis open sets $\{\Spec_{\fu}\!\! M_{a_i}\}$, and cover each of these with basis open sets $\{\Spec_{\fu}\!\! N_{b_{ij}}\}$. Because all coverings of affine schemes are trivial, we conclude that $\Spec_{\fu}\!\! N_{b_{ij}}$ equals $\Spec_{\fu}\!\! N$ for some couple $(i,j)$, and in particular $\Spec_{\fu}\!\! N$ equals $\Spec_{\fu}\!\! M_{a_i}$. The fact that (\ref{m2}) implies (\ref{m1}) comes from Corollary \ref{cor:locmonfp}, Proposition \ref{prop:locmonflat} and the universal property of localizations. We are then left to prove that (\ref{m1}) implies (\ref{m2}). By universal property, the map $\varphi$ splits over the monoid 
\[
\varinjlim_{a_i\in\varphi^{-1}(N^\times)} M_{a_i}=M_\mfp
\]
where $\mfp$ is $\varphi^{-1}(N\setminus N^\times)$. The  induced map $M_\mfp\ra N$ is local, and still an epimorphism. We now prove it is also flat. Suppose that $S$ is a $M_\mfp$-module. We claim that $S=S\otimes_M M_\mfp$. Indeed, the map $x\mapsto x\otimes 1$ defines an inverse of the natural map $x\otimes\frac{m}{f}\mapsto\frac{m}{f}\cdot x$.
Also, by the essential uniqueness of the adjoint functor, whenever we have a composite map of monoids $M\ra N\ra P$, then the functor $(\otimes_M N)\otimes_NP$ is canonically isomorphic to the functor $\otimes_MP$. We then write $S\otimes_MN\otimes_NP$ without using brackets, and consider it equal to $S\otimes_MP$, for any $M$-module $S$.
Now consider a finite limit $\lim S_i$ of $M_\mfp$-modules. We write $\hat{S}_i$ whenever we consider them as $M$-modules. Using the flatness of $\varphi$ and of localizations (Proposition \ref{prop:locmonflat}), we then conclude the following chain of isomorphisms
\[
\begin{aligned}
&(\lim S_i)\otimes_{M_\mfp}N=(\lim \hat{S}_i\otimes_M M_\mfp)\otimes_{M_\mfp}N=(\lim\hat{S}_i)\otimes_MM_\mfp\otimes_{M_\mfp}N=\\
&=(\lim\hat{S}_i)\otimes_MN=\lim(\hat{S}_i\otimes_MN)=\lim(\hat{S}_i\otimes_MM_\mfp\otimes_{M_\mfp}N)=\\
&=\lim(S_i\otimes_{M_\mfp}N)
\end{aligned}
\]
which proves that $M_\mfp\ra N$ is flat.

By Proposition \ref{prop:flatepifpmon}, we conclude that $M_\mfp\ra N$ is an isomorphism. Because of the finite presentation property, the identity map $N\ra M_\mfp$ has to split over some $M_a$ with $a\in\varphi^{-1}(N^\times)$. Because all the maps involved are maps of $M$-algebras, we conclude that $N=M_a$, as wanted.
 
\end{proof}

\begin{cor}\label{cor:zarimmon}
Let $\varphi\colon M\ra N$ be a map of monoids. The induced map $\Spec N\ra\Spec M$ is an open Zariski immersion in the sense of Definition \ref{df:openimm1} if and only if the induced map $\Spec_{\fu}\!\! N\ra\Spec_{\fu}\!\! M$ is an open Zariski immersion in the sense of Definition \ref{df:d-zaraff}.
\end{cor}

\begin{thm}
The Zariski site of affine geometrical $\fu$-schemes is equivalent to the Zariski site of $\Mon\op$.
\end{thm}

\begin{proof}
The two categories underneath are equivalent because of Proposition \ref{prop:specadjmon}. By the previous corollary, we also know that open immersions are the same. We have to prove that coverings are the same. Let $M$ be a monoid. In the case of affine geometrical $\fu$-schemes, coverings must include the trivial immersion $\Spec_{\fu}\!\!M\ra\Spec_{\fu}\!\!M$. We now prove that this is also true for the topology defined in \ref{df:openimm1}. Let $\{\Spec M_{a_i}\ra\Spec M\}$ be a Zariski covering. Suppose that none of these open immersions is trivial, i.e. that none of the $a_i$'s is invertible. Consider the $M$-module $M/\!\!\sim_\mfm$ where $\sim_\mfm$ identifies the non-invertible elements in $M$. We claim that $(M/\!\!\sim_\mfm)\otimes_MM_{a_i}$ is isomorphic to the trivial $M$-module $\{*\}$, for all $a_i$'s. Indeed, since $a_i$ is not invertible, we conclude the following sequence of equalities for any element $[x]\otimes\frac{m}{a_i^k}$ in $(M/\!\!\sim_\mfm)\otimes_MM_{a_i}$:
\[
[x]\otimes\frac{m}{a_i^k}=[mx]\otimes\frac{a_i}{a_i^{k+1}}=[a_i]\otimes\frac{1}{a_i^{k+1}}=[a_i^{k+1}a_i]\otimes\frac{1}{a_i^{k+1}}=[a_i]\otimes1.
\]
However, the morphism $(M/\!\!\sim_\mfm)\ra\{*\}$ is never an isomorphism, unless $M$ is the trivial group in which case the statement is obvious. We then conclude that any Zariski covering must include the trivial open immersion, as claimed.
 
\end{proof}

\begin{warning}
From now on, we will then drop the subscript when referring to affine geometrical $\fu$-schemes, and just write $\Spec M$. Also, we won't refer to any specific definition when considering open immersions of affine $\fu$-schemes. It is also legitimate to refer to the site we built on $\Mon\op$ as the \emph{Zariski site}, without specifying which definition we are using at every occurrence.
\end{warning}

\begin{lem}\label{lemma:aperti}
A map $X\ra Y$ of geometrical $\fu$-schemes is an open immersion if and only if for any affine scheme $\Spec M$ over $Y$, the induced arrow $X\times_Y\Spec M\ra \Spec M$ is an open immersion.
\end{lem}

\begin{proof}
This follows in the same way as in \cite{EGAI}, I.4.2.4.
 
\end{proof}

\begin{prop}\label{prop:openimmftv}
Let $f\colon\mcf\ra\mcg$ be a morphism of Zariski sheaves over $\Mon\op$, and let $\mcg=h_{\Spec M}$ be affine. Then $f$ is an open immersion if and only if $\mcf$ is isomorphic over $\mcg$ to $h_U\colonequals\Hom(\cdot,U)$ where $U$ is an open geometrical $\fu$-subscheme of $\Spec M$.
\end{prop}

\begin{proof}
By \cite{toenvaquie}, 2.14, this amounts to say that for a family of affine open geometrical $\fu$-subschemes $\{\Spec M_i\}$ of $\Spec M$, the image of the sheaf map \mbox{$\coprod{h_{\Spec M_i}}\ra h_{\Spec M}$} is $h_U$, where $U$ is the open geometrical $\fu$-subschemes constituted by the union of the $\Spec M_i$'s, and this is clear by \cite{mamo}, III.7.7.
 
\end{proof}

\begin{thm}\label{thm:main2}
The category of $\fu$-schemes is equivalent to the category of geometrical $\fu$-schemes.
\end{thm}

\begin{proof}
Since the category of monoidal spaces is cocomplete (Proposition \ref{prop:mscoc}), the inclusion $\Aff\ra\MS$ induces an adjoint pair $\PshAff\rightleftarrows\MS$ by means of \cite{schapira} Theorem 2.7.1, in which the left adjoint is the functor $|\cdot|\colon\PshAff\ra\MS$ that sends each object $\colim h_{\Spec M}$ to $\colim \Spec M$ and the right adjoint is the functor $h\colon\MS\ra\PshAff$ that sends $X$ to $h_X=\Hom(\cdot,X)$. Let now $X$ be a geometrical $\fu$-scheme, and let $\{\Spec M_i\ra X\}$ be an affine Zariski covering of it. Because the Zariski topology is subcanonical (Proposition \ref{prop:zarissub}), we conclude that $h_X$ is indeed a sheaf over $\Aff$. Fix now an affine $\fu$-scheme $h_{\Spec N}$ over $h_X$. By Lemma \ref{lemma:aperti}, the morphism $\Spec M_i\times_X\Spec N\ra\Spec N$ is an open immersion. Because of Definition \ref{df:openimm2}, Proposition \ref{prop:openimmftv}, and the fact that $h$ is a right adjoint, we can also conclude that the map 
\[
h(\Spec M_i\times_X\Spec N\ra\Spec N)=h_{\Spec M_i}\times_{h_X}h_{\Spec N}\ra h_{\Spec N}
\]
is an open immersion. This proves that each map $h_{\Spec M_i}\ra h_{\Spec M}$ is an open immersion.
Now we also prove that $\coprod h_{\Spec M_i}\ra h_X$ is an epimorphism. Indeed, let $\mcf$ be another sheaf, and let $f,g$ be maps from $h_{X}$ to $\mcf$ such that $f\varphi_i=g\varphi_i$ for every $i$. Note that, using \cite{SGAIV1} III.4, $\mcf$ can be seen not only as a sheaf over affines, but also as a sheaf over geometrical $\fu$-schemes. Hence, by Yoneda's lemma, the maps $f,g$ translate into two elements $\rho,\sigma$ in $\mcf(X)$ such that $\mcf(\varphi_i)(\rho)=\mcf(\varphi_i)(\sigma)$ for every $i$. Since $\mcf$ is a sheaf and because the $\varphi_i$'s define a covering, this implies that $\rho=\sigma$, hence $f=g$.
We then conclude that $h_X$ is a $\fu$-scheme. 

By the co-Yoneda lemma (\cite{mclane} X.6.3), we can write a presheaf of affines $\mcf$ as the colimit of the functor
\begin{align*}
\Aff_{/\mcf}&\ra\Psh(\cat)\\
(\Hom(\cdot,A)\ra \mcf)&\mapsto\Hom(\cdot,A).
\intertext{In particular, $|h_X|$ is the colimit of the functor}
\Aff_{/X}&\ra\MS\\
(A\ra X)&\mapsto A.
\end{align*}
Since affine geometrical $\fu$-schemes are dense in geometrical $\fu$-schemes, the colimit of this functor restricted to $\fu$-schemes is exactly $X$ (\cite{mclane}, X.6.2), hence there is a natural map $|h_X|\ra X$. We also know that $X$ is the colimit in $\MS$ of the gluing diagram induced by an affine open covering, which is embedded in the colimiting diagram \mbox{$\Aff_{/X}\ra \MS$}. Hence we have also a map $X\ra|h_X|$, which determines an isomorphism.

Now suppose that $\mcf$ is a $\fu$-scheme with an open affine covering $\{h_{\Spec M_i}\}$. Because $\fu$-schemes have fibered products (\cite{toenvaquie}, 2.18)
, we can also consider affine open coverings $\{h_{\Spec M_{ijk}}\}$ of the $\fu$-schemes $h_{\Spec M_i}\times_X h_{\Spec M_j}$. By \cite{mamo} IV.7.3 and \cite{mamo} A.1.1, an epimorphism of sheaves is the coequalizer of its kernel pair, and fiber products distribute over coproducts. Therefore, we conclude that $\mcf$ is the coequalizer in the diagram below.
\[
\coprod h_{\Spec M_i}\times_X h_{\Spec M_j}\rightrightarrows\coprod h_{\Spec M_i}\rightarrow\mcf
\]
Note that all these maps are open immersions. Indeed, by their very definition, open immersions are stable under affine base change, hence $h_{\Spec M_i}\times_\mcf h_{\Spec M_j}\ra h_{\Spec M_i}$ is an open immersion. In particular, by Proposition \ref{prop:openimmftv}, these maps can be written as $h_{U_{ij}}\ra h_{\Spec M_i}$ induced by open immersions $U_{ij}\ra\Spec M_i$.
We then conclude that $|\mcf|$ is the coequalizer of the diagram
\[
\coprod U_{ij}\rightrightarrows\coprod \Spec M_i\rightarrow|\mcf|
\]
so that it is a gluing of affines on open subsets, hence a geometrical $\fu$-scheme. By letting $\mcg$ be another $\fu$-scheme, we can also construct the equalizing diagram
\[
\Hom(\mcf,\mcg)\ra\coprod \Hom(h_{\Spec M_i},\mcg)\rightrightarrows\coprod \Hom(h_{\Spec M_i}\times_X h_{\Spec M_j},\mcg)
\]
and hence conclude that the Zariski topology on $\fu$-schemes is subcanonical. We can then define an inverse of the map $\mcf\ra h_{|\mcf|}$ by gluing the maps $h_{\Spec M}\ra h_{|\mcf|}$, hence $\mcf\cong h_{|\mcf|}$. This concludes the proof.
 
\end{proof}

It is easy to see that the equivalence of categories respects the topology of the two sites.

\begin{prop}\label{prop:main3}
A morphism of geometrical $\fu$-schemes is an open immersion if and only the induced morphism of $\fu$-schemes is an open immersion. Let now $X$ be a fixed geometrical $\fu$-scheme. A collection of geometrical $\fu$-schemes over $X$ is an open Zariski covering of $X$ if and only if the induced collection of $\fu$-schemes over $h_X$ is an open Zariski covering of $h_X$.
\end{prop}

\begin{proof}
The first claim follows from the fact that open coverings in both cases can be defined via affine base change (by using Lemma \ref{lemma:aperti} and Definition \ref{df:openimm2}), and in the affine case the two notions do agree. For coverings, it suffices to write down the associate coequalizing diagrams and use the gluing lemma.
 
\end{proof}

We remark that the proofs of Theorem \ref{thm:main2} and Proposition \ref{prop:main3} can be directly generalized to the context of schemes over $\zz$, providing an alternative proof of the equivalence presented in \cite{demazuregabriel}, I.1.4.4.

\section{Base change functors}

After having defined schemes over $\fu$, the natural question is how to lift them to classical schemes over $\zz$. We want to consider this process like a base change with $\zz$ over $\fu$. This can be done starting from the functor that lifts a monoid $M$ to the ring $\zz[M]$. However, the two approaches to $\fu$-geometry we presented in the past sections have different ways to generalize this functor to arbitrary schemes. Not surprisingly, Deitmar's definition (\cite{deitmar}, Section 2) is more ``geometric'', while To\"en-Vaquié's approach (\cite{toenvaquie}, Section 2.5) is more ``functorial''. Given that the two perspectives on schemes are equivalent, we have to prove that also the two ways of base-changing are naturally equivalent.

\begin{df}
The forgetful functor $\Ring\ra\Mon$ has a left adjoint $\Mon\ra\Ring$ that sends a monoid $M$ to the ring $\zz[M]$. We indicate this functor with the notation $\otimes_{\fu}\zz$. 
\end{df}

\begin{lem}\label{lemma:openbc}
Let $\Spec N\ra\Spec M$ be an open immersion of affine schemes over $\fu$. Then the induced map \[\Spec (N\otimes_{\fu}\zz)\ra\Spec( M\otimes_{\fu}\zz)\] is an open immersion of affine schemes over $\zz$.
\end{lem}

\begin{proof}
By Theorem \ref{thm:3opens}, it suffices to show that, for a given element $a\in M$, there is an isomorphism
\[
M_a\otimes_{\fu}\zz=\zz[M_a]\cong\zz[M]_a\!=\!(M\otimes_{\fu}\zz)_a
\]
where the second localization is taken in the category of rings. A map $\zz[M_a]\ra\zz[M]_a$ is induced by the map of monoids $M_a\ra\zz[M]_a$, which is in turn induced by the natural map $M\ra\zz[M]_a$. A map $\zz[M]_a\ra\zz[M_a]$ is induced by the map $\zz[M]\ra\zz[M_a]$, which is in turn induced by the natural map $M\ra M_a$. It is easy to see that these two maps are inverse one of the other.
 
\end{proof}

\begin{df}\label{df:bcdei}
Let $X$ be a geometrical $\fu$-scheme and let $\{\Spec M_i\}$ be an affine covering of it. Fix now an affine open covering $\{\Spec M_{ijk}\}$ for each $\Spec M_i\times_X\Spec M_j$. By Lemma \ref{lemma:openbc}, we can define a scheme over $\zz$ by gluing the affine schemes \mbox{$\Spec(M_i\otimes_{\fu}\zz)$} over \mbox{$\Spec(M_{ijk}\otimes_{\fu}\zz)$}. The scheme over $\zz$ we obtain is called \emph{base change of $X$, with respect to the covering $\{\Spec M_{ijk}\}$}.
\end{df}

\begin{df}
As described in \cite{toenvaquie}, Section 2.5, the adjoint couple from $\Mon$ to $\Ring$ induces a functor from Zariski sheaves on affine schemes over $\zz$ to Zariski sheaves on affine schemes over $\fu$, which has a left adjoint $\otimes_{\fu}\zz$. Also, the functor $\otimes_{\fu}\zz$ is such that $\fu$-schemes are mapped to schemes. Hence, its restriction defines a functor
\[
\begin{split}
\Sch_{\fu}&\ra\Sch\\
X&\mapsto X\otimes_{\fu}\zz.
\end{split}
\]
called the \emph{base change functor}.
\end{df}

\begin{prop}
Base change of geometrical $\fu$-schemes does not depend on the covering and is canonically equivalent to base change of $\fu$-schemes.
\end{prop}

\begin{proof}
We remark that the base change functor is automatically defined from the adjoint couple from $\Mon$ to $\Ring$. Let $X$ be an arbitrary scheme over $\fu$, and let $\{\Spec M_{ijk}\}$ be coverings as in Definition \ref{df:bcdei}. We can then write $X$ as the coequalizer of an affine diagram
\[
\coprod \Spec M_{ijk}\rightrightarrows\coprod \Spec M_i\rightarrow X.
\]
Since $\otimes_{\fu}\zz$ is a left adjoint, we conclude that $X\otimes_{\fu}\zz$ is the coequalizer of the diagram
\[
\coprod \Spec( M_{ijk}\otimes_{\fu}\zz)\rightrightarrows\coprod \Spec (M_i\otimes_{\fu}\zz)\rightarrow X\otimes_{\fu}\zz
\]
which is exactly the image of $X$ via base change with respect to the fixed covering.
 
\end{proof}

We can hence summarize what we have done by saying that the part of the $\fu$-map in \cite{penalorscheid} that concerns Deitmar's and To\"en-Vaquié's schemes is correct, in the sense that both the equivalence between the two notions and the commutativity of the base change functors have been proven.

\section*{Acknowledgments}
I am deeply grateful to Professor Luca Barbieri Viale, who firstly introduced me to the ``fun of $\fu$'', and constantly encourages and enriches me with advices and teachings. I also thank Professors Bas Edixhoven and Ieke Moerdijk, whom I was honored to meet during my stay in the Netherlands, and who gave me specific support for the needs of this paper. I also thank Professor Bertrand To\"en who answered to my questions with courtesy and clarity, and who indicated to me the work of Florian Marty, and Florian Marty himself for sharing his interesting material with me, as well as his ideas on the field. 
A specific workshop of young researchers was also organized in Granada, 2009 in order to resume what had been done so far in the field of \mbox{$\fu$-geometry}, and jot down a plan for the future. I express my gratitude to the organizers and the participants for having invited me to such event, and enlightened me with interesting lectures and conversations. I would like to thank specifically Peter Arndt for having pointed out to me the result \cite{schapira} 2.7.1, and for having elucidated many other ideas. 
I also had the chance to talk about many aspects of this article in a seminar talk organized in the University of Milan. I warmly thank the organizers for this opportunity, and the participants for their interesting remarks. Specifically, I would like to thank Professor Fabrizio Andreatta, who has also been a enormous source of help and insight.

\medskip

The final publication is available at springerlink.com

\end{document}